\documentclass{amsart}
\usepackage{amssymb}
\usepackage{amsmath}
\usepackage{a4}

\newcommand{\J}{\mathcal{J}}

\begin{document}
\newtheorem{theorem}{Theorem}[section]
\newtheorem{lemma}[theorem]{Lemma}
\newtheorem{remark}[theorem]{Remark}
\newtheorem{definition}[theorem]{Definition}
\newtheorem{corollary}[theorem]{Corollary}
\newtheorem{example}[theorem]{Example}
\def\MM{\mathfrak{M}}
\def\spec{\operatorname{Spec}}
\def\refBBZa01a{\cite{refBBZa01a}}
\def\refBGVa05{\cite{BGVa05}}
\def\refDGV05{\cite{DGV05}}
\def\refGKVm02{\cite{GKVm02}}
\def\refGRVa05{\cite{GRVa05}}
\def\qedbox{\hbox{$\rlap{$\sqcap$}\sqcup$}}
\makeatletter
  \renewcommand{\theequation}{%
   \thesection.\alph{equation}}
  \@addtoreset{equation}{section}
 \makeatother
\def\BB{\mathcal{B}}
\title{Four-dimensional Osserman metrics of neutral signature}
\author{E. Garc\'{\i}a-R\'{\i}o, P. Gilkey, M. E. V\'{a}zquez-Abal\\ and R.
V\'{a}zquez-Lorenzo }
%\date{Version 4d as of 28 March 2008 changes marked blue}%version
\begin{address}{EGR, MEVA, and RVL: Department of Geometry and Topology, Faculty of
Mathematics, University of Santiago de Compostela, 15782 Santiago de Compostela, Spain}
\end{address}
\begin{email}{xtedugr@usc.es, meva@zmat.usc.es, ravazlor@usc.es}\end{email}
\begin{address}{PG: Mathematics Department, University of Oregon, Eugene, OR
97403, USA and Max Planck Institute in the Math. Sciences, Leipzig Germany}\end{address}
\begin{email}{ gilkey@uoregon.edu}\end{email}

\begin{abstract}In the algebraic
context, we show that null Osserman, spacelike Osserman, and timelike Osserman
are equivalent conditions for a model of signature (2,2). We also
classify the null Jordan Osserman models  of signature (2,2). In
the geometric context, we show that a pseudo-Riemannian manifold of signature
(2,2) is null Jordan Osserman if and only if either it has constant sectional
curvature or it is locally a complex space form.
\end{abstract}
\maketitle

\section{Introduction}\label{sect-1} Let $\mathcal{M}:=(M,g)$ be a pseudo-Riemannian manifold.
We say that a tangent vector $v$ is {\it spacelike}, {\it timelike}, or {\it
null} if $g(v,v)>0$, if $g(v,v)<0$, or if $g(v,v)=0$, respectively. Geometric
properties derived from conditions on spacelike, timelike and null vectors
can have quite different meanings. For instance, the notions of spacelike,
timelike and null geodesic completeness are non-equivalent and independent
conditions. Although spacelike and timelike conditions can sometimes become
equivalent (for example as concerns boundedness conditions on the
sectional  curvature), they can be quite different than similar null
conditions, which are sometimes related to the conformal geometry of the
manifold.

Let $R(x,y):=\nabla_x\nabla_y-\nabla_y\nabla_x-\nabla_{[x,y]}$ be the curvature
operator of $\mathcal{M}$. The associated {\it Jacobi operator}
$\J_R(x):y\rightarrow R(y,x)x$ encodes much geometric information concerning
the manifold. One has that $\J_R(\lambda v)=\lambda^2\J_R(v)$; this rescaling
property plays a crucial role. Let $S^\pm(\mathcal{M})$ be the unit sphere
bundles of spacelike and timelike unit tangent vectors in $M$ and let
$N(\mathcal{M})$ be the null cone of non-zero null vectors. One says that
$\mathcal{M}$ is spacelike (resp. timelike) Osserman if the eigenvalues of
$\J_R$ are constant on $S^+(\mathcal{M})$ (resp. on $S^-(\mathcal{M}))$.
Normalizing the length of the tangent vector to be $\pm1$ takes into account
the scaling of the Jacobi operator $\J_R(\lambda v)=\lambda^2\J_R(v)$ noted
above. Perhaps somewhat surprisingly, spacelike Osserman and timelike Osserman are equivalent conditions
\cite{GKVaVl,Gilkey1}.

We shall say that $\mathcal{M}$ is {\it null Osserman} if the eigenvalues of
$\J_R$ are constant on the null cone $N(\mathcal{M})$; with this definition, if $\mathcal{M}$ is null
Osserman, then necessarily $\J_R(v)$ is nilpotent if $v\in N(\mathcal{M})$ and $\J_R(v)$ has only the
eigenvalue $0$. Any spacelike or timelike Osserman manifold is necessarily null Osserman;
the converse can fail in general -- see, for example,
\cite{GKVa} in the Lorentzian setting.

The Jordan normal form plays a crucial role in the higher signature setting -- a self-adjoint linear
transformation need not be determined by its eigenvalues if the metric in question is indefinite. One
says that $\mathcal{M}$ is spacelike, timelike, or null Jordan Osserman if the Jordan normal form of
$\J_R(\cdot)$ is constant on $S^+(\mathcal{M})$, on $S^-(\mathcal{M})$, or on $N(\mathcal{M})$,
respectively. It is known \cite{Gilkey1,G-I,G-I-2} that spacelike and timelike Jordan Osserman are
inequivalent conditions; further neither necessarily implies the null
Jordan Osserman condition.

In this paper, we concentrate on the $4$-dimensional setting. Chi
\cite{Chi} showed that any Riemannian Osserman $4$-manifold is
locally isometric to a $2$-point homogeneous space; it follows from
later work  \cite{BBG,GKVa} that any Lorentzian $4$-manifold has
constant sectional curvature. However the situation is much more
complicated in neutral signature $(2,2)$; there exist many  examples
of nonsymmetric Osserman pseudo-Riemannian manifolds of neutral
signature -- see \cite{DGV} and \cite{GR-VA-VL}. Indeed, despite the
results  of \cite{ABBR, BBR,DGV1,GV},  the general problem of
obtaining a complete description of $4$-dimensional  Osserman
metrics of neutral signature remains open.

It is convenient to work algebraically. Let $V$ be a finite dimensional
real vector space which is equipped with a non-degenerate symmetric bilinear form
$\langle\cdot,\cdot\rangle$ of signature $(p,q)$. Let $A\in\otimes^4(V^*)$ be an algebraic curvature
tensor on $V$, i.e. a tensor which has the symmetries of the Riemann curvature tensor:
\begin{eqnarray*}
&&A(x,y,z,v)=-A(y,x,z,v)=A(z,v,x,y),\\
&&A(x,y,z,v)+A(y,z,x,v)+A(z,x,y,v)=0\,.
\end{eqnarray*}
This defines a model $\MM:=(V,\langle\cdot,\cdot\rangle,A)$. We shall often prove results on the
algebraic level (i.e. for models), and then obtain corresponding conclusions in the geometric
context. The notions spacelike unit vector, timelike unit vector, null vector, Jacobi operator, etc.
extend naturally to this setting.

\subsection{Null Osserman algebraic curvature tensors}\label{sect-1.1}
Henceforth, let $\langle\cdot,\cdot\rangle$ be an inner product of signature $(2,2)$ on a $4$-dimensional real vector space $V$.
Fix an orientation of $V$ and let $\BB=\{e_1,e_2,e_3,e_4\}$ be an oriented orthonormal basis for $V$ where $e_1$ and
$e_2$ are timelike and where $e_3$ and $e_4$ are spacelike.

At the algebraic level, in signature $(2,2)$ the conditions spacelike Osserman,
timelike Osserman, spacelike Jordan Osserman and timelike Jordan Osserman are equivalent to the condition that $\MM$ is Einstein and
self-dual with respect to a suitably chosen local orientation  \cite{ABBR,GKV}. In Section \ref{sect-2}, we will
establish the following result which shows that these conditions are also equivalent to null Osserman:

\begin{theorem}\label{thm-1.1}
Let $\MM$ be a model of neutral signature $(2,2)$. Then the following conditions are
equivalent:
\begin{enumerate}
\item $\MM$ is spacelike Osserman.
\item $\MM$ is timelike Osserman.
\item $\MM$ is spacelike Jordan Osserman.
\item $\MM$ is timelike Jordan Osserman.
\item $\MM$ is Einstein and self-dual for a suitably chosen local orientation.
\item $\MM$ is null Osserman.
\end{enumerate}\end{theorem}

\begin{remark}\label{rmk-1.2} \rm The action of homotheity on the null vectors is a central one in this
subject and it is worth saying a few extra words concerning this. With our definition, it is immediate
that $\MM=(V,\langle\cdot,\cdot\rangle,A)$ is null Osserman implies that $0$ is the only eigenvalue of $\mathcal{J}_A$
on $N(V,\langle\cdot,\cdot\rangle)$. There is, although, an alternate, and different,
formulation one could use. One says that $\MM$ is {\it
projectively null Osserman} if either $\MM$ is null Osserman or if given $0\ne n_1,n_2\in N(V,\langle\cdot,\cdot\rangle)$,
there is a non-zero constant $\lambda$ so that
$\operatorname{Spec}(\mathcal{J}_A(n_1))=\lambda\operatorname{Spec}(\mathcal{J}_A(n_2))$. We refer to \cite{BGNS} for related work;
we only introduce this concept for the sake of completeness as it plays no role in our development.\end{remark}

\subsection{Null Jordan Osserman algebraic curvature tensors}\label{sect-1.2}
There are two algebraic curvature tensors which will play
a distinguished role in our development. If
$\Psi$ is an anti-symmetric endomorphism of $V$, define the associated algebraic curvature tensor $A^\Psi$ by setting:
\begin{equation}\label{eqn-1.a}
A^\Psi(x,y,z,v):=\langle\Psi y,z\rangle\langle\Psi x,v\rangle
-\langle\Psi x,z\rangle\langle\Psi y,v\rangle-2\langle\Psi x,y\rangle\langle\Psi z,v\rangle\,.
\end{equation}
Such tensors span the linear space of all algebraic curvature tensors \cite{F03}.

The sectional curvature of a non-degenerate $2$-plane $\pi=\operatorname{Span}\{x,y\}$ is given by:
$$K_A(\pi):=\frac{A(x,y,y,x)}{\langle x,x\rangle\langle y,y\rangle-\langle x,y\rangle\langle x,y\rangle}\,;$$
$A$ has {\it constant sectional curvature} $\kappa_0$ if and only if $A=\kappa_0A^0$ where $A^0$ is the algebraic curvature
tensor of constant sectional curvature $+1$ defined by:
\begin{equation}\label{eqn-1.b}
A^0(x,y,z,v):=\langle y,z\rangle\langle x,v\rangle-\langle x,z\rangle\langle y,v\rangle\,.
\end{equation}
We note for future reference that Equations (\ref{eqn-1.a}) and (\ref{eqn-1.b}) imply that:
\begin{equation}\label{eqn-1.c}
\mathcal{J}_{A^\Psi}(x):y\rightarrow3\langle y,\Psi x\rangle\Psi x\quad\text{and}\quad
\mathcal{J}_{A^0}(x):y\rightarrow\langle x,x\rangle y-\langle x,y\rangle x\,.
\end{equation}

Assume that $\Psi$ is skew-adjoint. We say that
$\Psi$ is an {\it orthogonal complex structure} if $\Psi^2=-\operatorname{id}$ and that $\Psi$ is an {\it adapted paracomplex
structure} if
$\Psi^2=\operatorname{id}$. We say that a triple of skew-adjoint operators $\{\Psi_1,\Psi_2,\Psi_3\}$ is a
{\it paraquaternionic structure} if $\Psi_1^2=-\operatorname{id}$, $\Psi_2^2=\operatorname{id}$, $\Psi_3^2=\operatorname{id}$, and if
$\Psi_i\Psi_j+\Psi_j\Psi_i=0$ for $i\ne j$. We can define a paraquaternionic structure by setting:
\begin{equation}\label{eqn-1.d}
\begin{array}{llll}
\Psi_1e_1=-e_2,&\Psi_1e_2=\phantom{-}e_1,&\Psi_1e_3=\phantom{-}e_4,&\Psi_1e_4=-e_3,\\
\Psi_2e_1=\phantom{-}e_3,&\Psi_2e_2=\phantom{-}e_4,&\Psi_2e_3=\phantom{-}e_1,&\Psi_2e_4=\phantom{-}e_2,\\
\Psi_3e_1=\phantom{-}e_4,&\Psi_3e_2=-e_3,&\Psi_3e_3=-e_2,&\Psi_3e_4=\phantom{-}e_1.
\end{array}\end{equation}
Note that $\Psi_3=\Psi_1\Psi_2$. If $\{\bar\Psi_1,\bar\Psi_2,\bar\Psi_3\}$ is any other paraquaternionic structure on $V$, there exists
an isometry $\phi$ of $V$ so $\phi^*\bar\Psi_1=\Psi_1$, $\phi^*\bar\Psi_2=\Psi_2$, and $\phi^*\bar\Psi_3=\pm\Psi_3$; this
slight sign ambiguity will play no role in our constructions.

Let $x$ be a spacelike or timelike vector. Then there is an orthogonal direct sum decomposition $V=x\cdot\mathbb{R}\oplus x^\perp$.
Since $\mathcal{J}_A(x)x=0$, $\mathcal{J}_A(x)$ preserves $x^\perp$. There are four different possibilities which describe the Jordan
normal form of $\mathcal{J}_A(x)$ restricted to $x^\perp$, we refer to \cite{BBR,GKV} for further details:
\begin{equation}\label{eqn-1.e}
\begin{array}{l}
\left(\begin{array}{ccc}
   \alpha & 0 & 0
   \\
   0 & \beta & 0
   \\
   0 & 0 & \gamma
\end{array}\right),\
\left(\begin{array}{ccc}
   \alpha & -\beta & 0
   \\
   \beta & \alpha & 0
   \\
   0 & 0 & \gamma
\end{array}\right),\
\left(\begin{array}{ccc}
   \beta & 0 & 0
   \\
   0 & \alpha & 0
   \\
   0 & 1 & \alpha
\end{array}\right),\
\left(\begin{array}{ccc}
   \alpha & 0 & 0
   \\
   1 & \alpha & 0
   \\
   0 & 1 & \alpha
\end{array}\right).
\\
\qquad\emph{Type Ia}\qquad\qquad\emph{Type Ib}\qquad\quad
\qquad\emph{Type II}\qquad\qquad \emph{Type III}
\end{array}
\end{equation}
Type Ia corresponds to a diagonalizable operator, Type Ib to an operator with a complex
eigenvalue and Type II (resp. Type III) to a double (resp. triple) root of the minimal polynomial of the operator. If $\MM$ is
spacelike, timelike, or null Osserman, then the Jordan normal form of $\mathcal{J}_A$ is constant on the spacelike and timelike unit
vectors and we classify $A$ according to the $4$-Types above. In Section \ref{sect-3}, we construct, up to
isomorphism, all the spacelike Jordan Osserman algebraic curvature tensors and perform the analysis necessary to
establish the following classification result:

\begin{theorem}\label{thm-1.3}
Let $\MM:=(V,\langle\cdot,\cdot\rangle,A)$ be a model of signature $(2,2)$. Then $\MM$ is null
Jordan Osserman if and only if $A$ is of Type Ia and one of the following holds:
\begin{enumerate}
\item There exists a constant $\kappa_0$ so that $A=\kappa_0A^{0}$.
\item There exists constants $\kappa_0$ and $\kappa_J$ with $\kappa_J\ne0$ so that $A=\kappa_0
A^0+\kappa_JA^J$ where
$J$ is an orthogonal complex structure on
$V$.
\item There exists a constant $\kappa_P\ne0$ so that $A=\kappa_PA^P$ where $P$ is  an adapted
paracomplex structure on
$V$t.
\item There exist constants $\kappa_1,\kappa_2,\kappa_3$ so that $\kappa_2\kappa_3
(\kappa_2+\kappa_1)(\kappa_3+\kappa_1)>0$, so that the associated eigenvalues
$\{3\kappa_1,-3\kappa_2,-3\kappa_3\}$ are all distinct, and so that $A=\kappa_1 A^{\Psi_1} +
\kappa_2 A^{\Psi_2} +
\kappa_3 A^{\Psi_3}$ where $(\Psi_1,\Psi_2,\Psi_3)$ is a paraquaternionic structure on $V$.
\end{enumerate}
\end{theorem}

\begin{remark}\label{re:nuevo_1}\rm
The inequality $\kappa_2\kappa_3
(\kappa_2+\kappa_1)(\kappa_3+\kappa_1)>0$ is equivalent to the fact
that the cross ratio
\[
(0,\kappa_1,-\kappa_3,-\kappa_2)=\frac{\kappa_3(\kappa_2+\kappa_1)}{\kappa_2(\kappa_3+\kappa_1)}>0.
\]
Let $\mathbb{S}^2$ be the unit sphere in $\mathbb{R}^3$. This inequality is equivalent to the
fact that the set of points
$(0,-\kappa_3,-\kappa_2)$  and $(\kappa_1,-\kappa_3,-\kappa_2)$ give
the corresponding circles in $\mathbb{S}^2$ (via the stereographic
projection) the same orientation \cite{Marden}.
\end{remark}

\subsection{Null Jordan Osserman manifolds}\label{sect-1.3}
We characterize those neutral signature
$4$-manifolds which are null Jordan Osserman; null Osserman and null Jordan Osserman are not
equivalent conditions as the analysis of Section \ref{sect-3.4} shows.  We say that $\mathcal{M}$ is locally a complex space form if it
is an indefinite K\"{a}hler manifold  of constant
 holomorphic sectional curvature. We will use Theorem \ref{thm-1.3} to establish the following geometric result in Section
\ref{sect-4}:
\begin{theorem}\label{thm-1.4}
Let $\mathcal{M}$ be a connected pseudo-Riemannian manifold of neutral signature
$(2,2)$. Then $\mathcal{M}$ is null Jordan Osserman if and only either $\mathcal{M}$ has constant sectional curvature or $\mathcal{M}$
is locally a complex space form.
\end{theorem}

\begin{remark}\label{re:nuevo_2}\rm
Recall that there is another family of four-dimensional Osserman
manifolds with diagonalizable Jacobi operator: the paracomplex space
forms \cite{BBR}. Although the geometry of complex and paracomplex
space forms is very similar, the Jordan-Osserman condition
distinguishes them. So far, up to our knowledge, this is the first
algebraic curvature condition which  distinguishes between these two geometries.

\end{remark}

\section{null Osserman models of signature $(2,2)$}\label{sect-2}

We will work in the algebraic context to prove Theorem \ref{thm-1.1}. Here is a brief outline to this section. Previous work
establishes that Assertions (1)-(5) are equivalent. In Section \ref{sect-2.1}, we introduce various
notational conventions and show that spacelike Osserman models are null Osserman and that null Osserman
models are Einstein. Thus to complete the proof, it suffices to show null Osserman models are self-dual
or anti-self-dual. In Section \ref{sect-2.2}, we examine Einstein models. Lemma \ref{lem-2.2} describes
the Weyl curvature operators in that setting and Lemma \ref{lem-2.3} gives an alternate
characterization of self-duality for an Einstein model. We use Lemma \ref{lem-2.3} to complete the
proof of Theorem \ref{thm-1.1} in Section
\ref{sect-2.3}.

\subsection{Notational conventions}\label{sect-2.1} Let
$\MM:=(V,\langle\cdot,\cdot\rangle,A)$ be a neutral signature $4$-dimensional model. We use the inner
product to raise indices and to define an associated Jacobi operator
$\J_A$, which is characterized by the identity:
$$
\langle\J_A(x)y,z\rangle=A(y,x,x,z)\,.
$$
Let $\BB=\{e_1,e_2,e_3,e_4\}$ be an oriented orthonormal basis for $V$. Let $g_{ij}:=\langle e_i,e_j\rangle$ and let $g^{ij}$ be the
inverse matrix. The associated {\it Ricci tensor} $\rho_A$, the {\it scalar curvature} $\tau_A$, and the {\it Weyl tensor} $W_A$ are
then defined by setting:
\begin{eqnarray*}
    &&\rho_A(x,y):=\sum_{i,j=1}^4g^{ij} A(e_i,x,y,e_j),
    \quad
    \tau_A:=  \sum_{i,j=1}^4 g^{ij} \rho_A(e_i,e_j),\\
    &&W_A(x,y,z,v):= A(x,y,z,v)
    + \textstyle\frac16\tau_A\{\langle y,z\rangle\langle x,v\rangle-\langle x,z\rangle\langle y,v\rangle\}\\
    &&\quad-\textstyle\frac{1}{2} \{\rho_A(y,z)\langle x,v\rangle-\rho_A(x,z)\langle y,v\rangle\\
    && \qquad+\rho_A(x,v)\langle y,z\rangle - \rho_A(y,v)\langle x,z\rangle\}\,.
\end{eqnarray*}
Let $A_{ijkl}=A_{ijkl}^\BB := A(e_i,e_j,e_k,e_l)$ denote the components
of
$A$ with respect to $\BB$ where $1\le i,j,k,l\le 4$; we shall drop the dependence on $\BB$ from the notation when there is no
danger of confusion. Let
$\{e^1,...,e^4\}$ be the dual basis for $V^*$.  The Hodge operator $\star:\Lambda^p(V^*)\rightarrow\Lambda^{4-p}(V^*)$ is characterized
by the identity:
$$\phi_p\wedge\star\theta_p=\langle \phi_p,\theta_p \rangle e^1\wedge e^2\wedge e^3\wedge e^4\,.$$
Thus, in particular,
$$\begin{array}{lll}
\star(e^1\wedge e^2)=e^3\wedge e^4,&\star(e^1\wedge e^3)=e^2\wedge e^4,&\star(e^1\wedge e^4)=-e^2\wedge e^3,\\
\star(e^2\wedge e^3)=-e^1\wedge e^4,&\star(e^2\wedge e^4)=e^1\wedge e^3,&\star(e^3\wedge e^4)=e^1\wedge e^2\,.
\end{array}$$
A crucial feature of $4$-dimensional geometry now enters. Since $\star^2=\operatorname{id}$, $\star$ induces a splitting of the space
of $2$-forms
$\Lambda^2(V^*)=\Lambda^+\oplus\Lambda^-$, where  $\Lambda^+$ and
$\Lambda^-$ denote the spaces of {\it self-dual} and {\it anti-self-dual}
two-forms
\[
    \Lambda^+=\{\alpha\in \Lambda^2:\star\alpha=\alpha\},
    \qquad \Lambda^-=\{\alpha\in \Lambda^2:\star\alpha=-\alpha\}.
\]
We have orthonormal bases $\left\{ E_1^\mp,
E_2^\mp, E_3^\mp\right\}$ for $\Lambda^\mp$ which are given by:
\begin{eqnarray*}
&&E_1^\mp = {\textstyle\frac1{\sqrt2}}(e^1\wedge e^2 \mp  e^3\wedge e^4),\quad
 E_2^\mp = {\textstyle\frac1{\sqrt2}}(e^1\wedge e^3 \mp  e^2\wedge e^4),\\
&&E_3^\mp ={\textstyle\frac1{\sqrt2}}(e^1\wedge e^4 \pm  e^2\wedge e^3),
\end{eqnarray*}
where the induced inner product on $\Lambda^\mp$ has signature $(2,1)$:
$$\langle E_1^\mp, E_1^\mp \rangle =1,\quad\langle E_2^\mp,
E_2^\mp \rangle =-1,\quad\langle E_3^\mp, E_3^\mp \rangle =-1\,.$$

Let $W^\mp_A$ be the restriction of
$W_A$  to the spaces $\Lambda^\mp$;  $W_A^\mp:\Lambda^\mp \longrightarrow
\Lambda^\mp$. Then $\MM$ is said to be
\emph{self-dual} (resp., \emph{anti-self-dual}) if $W_A^-=0$ (resp. if $W_A^+=0$).

\begin{lemma}\label{lem-2.1}
Let $\MM=(V,\langle\cdot,\cdot\rangle,A)$ be a model of signature $(2,2)$.
\begin{enumerate}
\item If $\MM$ is spacelike Osserman, then $\MM$ is null Osserman.
\item If $\MM$ is null Osserman, then $\MM$ is Einstein.
\end{enumerate}
\end{lemma}

\begin{proof} Let $\MM$ be spacelike Osserman. Set $T_j(v):=\operatorname{Tr}\{\J_A(v)^j\}$.
Since the eigenvalues of $\J_A$ are constant on $S^+(V,\langle\cdot,\cdot\rangle)$, there are constants $c_j$ so that
$T_j(v)=c_j$ for $v\in S^+(V,\langle\cdot,\cdot\rangle)$. Since $T_j(\lambda v)=\lambda^{2j}T_j(v)$, we have
$T_j(v)=c_j\langle v,v\rangle^j$ for $v$ spacelike. Since the spacelike vectors form an open subset of
$V$, this polynomial identity holds for all $v\in V$. Thus, in particular, $T_j(v)=0$ if $v\in
N(V,\langle\cdot,\cdot\rangle)$. This implies that $0$ is the only eigenvalue of $\J_A(v)$ and shows
$\MM$ is null Osserman.

Suppose that $\MM$ is null Osserman. Let $s_1$ and $s_2$ be spacelike unit vectors. We may
choose a unit timelike vector $t$ which is perpendicular to $s_1$ and $s_2$. Let $n_i^\pm:=s_i\pm t$ be null vectors. Thus
$0=\operatorname{Tr}(\J_A(n_i^\pm))=\rho_A(n_i^\pm,n_i^\pm)$, and
$$0=\rho_A(s_i\pm t,s_i\pm t)=\rho_A(s_i,s_i)+\rho_A(t,t)\pm2\rho_A(s_i,t)\,.$$
This implies $\rho_A(s_i,t)=0$ and $\rho_A(s_i,s_i)+\rho_A(t,t)=0$;  in particular, one
has that $\rho_A(s_1,s_1)=-\rho_A(t,t)=\rho_A(s_2,s_2)$. Consequently, after rescaling, there is a constant $c$ so
$\rho_A(s,s)=c\langle s,s\rangle$ for every spacelike vector
$s$; this polynomial identity then continues to hold for all $s\in V$. Polarizing this identity then yields
$\rho_A=c\langle\cdot,\cdot\rangle$ and hence $\MM$ is Einstein.\end{proof}

\subsection{The Weyl tensor for an Einstein algebraic curvature tensor}\label{sect-2.2}
Let
\begin{eqnarray*}
&&\sigma_1 = 2A_{1212}+3\varepsilon
A_{1234}+A_{1313}+A_{1414},\\
&&\sigma_2 = A_{1212}+2
A_{1313}+3\varepsilon A_{1324}-A_{1414},\\
&&\sigma_3 =
A_{1212}+3\varepsilon A_{1234}-A_{1313}-3\varepsilon
A_{1324}+2A_{1414}\,.\end{eqnarray*}
The following lemma is now immediate:
\begin{lemma}\label{lem-2.2}
If $\MM$ is Einstein, then the self-dual Weyl curvature operator $W_A^+$ ($\varepsilon=1$) and the
anti-self-dual Weyl curvature operator $W_A^-$ ($\varepsilon=-1$) are given by:
$$\left(
   \begin{array}{ccc}
      \frac{\sigma_1}{3} & A_{1213}+\varepsilon A_{1224}  & A_{1214}-\varepsilon A_{1223}
      \\
      -A_{1213}-\varepsilon A_{1224} & -\frac{\sigma_2}{3}  & -A_{1314}+\varepsilon A_{1323}
      \\
      -A_{1214}+\varepsilon A_{1223} & -A_{1314}+\varepsilon A_{1323} & -\frac{\sigma_3}{3}
   \end{array}
   \right)\,.$$\end{lemma}

We now come to an observation which is of interest in its own right:

\begin{lemma}\label{lem-2.3}
If $\MM$ is Einstein, then the model $\MM$ is
anti-self-dual if and only if  $A^\BB_{1214}-A^\BB_{1223}=0$ for every oriented orthonormal frame $\BB$.
\end{lemma}

\begin{proof} If $\MM$ is anti-self-dual, we set $\varepsilon=1$ in Lemma \ref{lem-2.2} to see
$A^\BB_{1214}-A^\BB_{1223}=0$. Conversely, suppose $A^\BB_{1214}-A^\BB_{1223}=0$ for every $\BB$. Define a new basis
$\tilde\BB$ by setting $\tilde e_1=e_1$,
$\tilde e_2=e_2$,
$\tilde e_3=e_4$, and $\tilde e_4=-e_3$. We then have
$$
0=-A^{\tilde\BB}_{1214}+A^{{\tilde\BB}}_{1223}=A^\BB_{1213}+A^\BB_{1224}\,.
$$
Next, define $\tilde\BB$ by setting $\tilde e_1=e_1$, $\tilde e_2=\cosh\theta e_2+\sinh\theta e_3$, $\tilde
e_3=\sinh\theta e_2+\cosh\theta e_3$, and
$\tilde e_4=e_4$. This yields the relation:
$$
0=-A^{\tilde\BB}_{1214}+A^{\tilde\BB}_{1223}
=\cosh\theta\{-A^\BB_{1214}+A^\BB_{1223}\}+\sinh\theta\{-A^\BB_{1314}+A^\BB_{1323}\}\,.
$$
This shows $-A^\BB_{1314}+A^\BB_{1323}=0$. Thus, by Lemma \ref{lem-2.2},
$$W_A^+=
\frac13\left(\begin{array}{lll}\sigma_1^\BB&0&0\\0&-\sigma_2^\BB&0\\0&0&-\sigma_3^\BB\end{array}\right)\,.$$ Again
setting $\tilde e_1=e_1$, $\tilde e_2=\cosh\theta e_2+\sinh\theta e_3$, $\tilde e_3=\sinh\theta
e_2+\cosh\theta e_3$, and
$\tilde e_4=e_4$ yields bases for $\Lambda^\pm$ in the form
$$\tilde E_1^\pm=\cosh\theta E_1^\pm+\sinh\theta E_2^\pm,\quad
  \tilde E_2^\pm=\cosh\theta E_2^\pm+\sinh\theta E_1^\pm,\quad
  \tilde E_3^\pm=E_3^\pm\,.
$$
We may compute
\begin{eqnarray*}
W_A^+\tilde E_1^+&=&\sigma_1^{\tilde\BB}\tilde E_1^+=\sigma_1^{\tilde\BB}(\cosh\theta E_1^++\sinh\theta E_2^+)\\
&=&W_A^+(\cosh\theta E_1^++\sinh\theta E_2^+)=\sigma_1^\BB\cosh\theta E_1^+-\sigma_2^\BB\sinh\theta E_2^+\,.
\end{eqnarray*}
This shows $\sigma^{\tilde\BB}_1=\sigma_1^\BB=-\sigma_2^\BB$. A similar argument applied to the basis $\tilde e_1=e_1$, $\tilde
e_2=\cosh\theta e_2+\sinh\theta e_4$,
$\tilde e_3=e_3$, and $\tilde e_4=\sinh\theta e_2+\cosh\theta e_4$ yields $\sigma_1^\BB=-\sigma_3^\BB$. Since
$\sigma_1^\BB-\sigma_2^\BB-\sigma_3^\BB=0$, it now follows that $W_A^+=0$.
\end{proof}

\subsection{The proof of Theorem \ref{thm-1.1}}\label{sect-2.3}
Let $\MM$ be a null Osserman model. By Lemma \ref{lem-2.1}, $\MM$ is Einstein. We complete the proof of Theorem \ref{thm-1.1} by
showing $\MM$ is self-dual or anti-self-dual. Suppose the contrary and argue for a contradiction.
As $\MM$ is null Osserman, $\mathcal{J}_A$ is nilpotent so the characteristic polynomial $p_\lambda(\mathcal{J}_A(u))=\lambda^4$.
Let
$$
\begin{array}{l}
\mathcal{E}_1:=A_{1212}+2A_{1214}-2A_{1223}+2A_{1234}-A_{1324}+A_{1414},\\
Q(a,b):= (A_{1212}-2A_{1214}-2A_{1223}-2A_{1234}+A_{1324}+A_{1414}) a^4\vphantom{\vrule height 10pt}\\
\qquad+ (A_{1212}+2A_{1214}+2A_{1223}-2A_{1234}+A_{1324}+A_{1414}) b^4\\
\qquad+ 2 (A_{1212}+2A_{1313}-3A_{1324}-A_{1414}) a^2 b^2\\
\qquad+ 4 (A_{1213}-A_{1224}-A_{1314}-A_{1323}) a^3 b\\
\qquad+ 4 (A_{1213}-A_{1224}+A_{1314}+A_{1323}) a b^3.
   \end{array}
$$
If we take $u=a e_1 + b e_2 + a e_3 + b e_4$, then we have
$$ \lambda^4=p_\lambda(\J_A(u)) = \lambda^2
      \left(\lambda^2 -   Q(a,b) \mathcal{E}_1\right)\,.$$
As $p_\lambda(\J_A(u))=\lambda^4$, either
$Q(a,b)=0$ or $\mathcal{E}_1=0$. If we suppose that $\mathcal{E}_1\ne0$, we then have $Q(a,b)$ vanishes identically for all $a,b$.
This leads to the relations:
$$
   \begin{array}{l}
      A_{1213}-A_{1224}=0, \quad  A_{1214}+A_{1223}=0, \quad A_{1314}+A_{1323}=0,
      \\
      A_{1234}+A_{1313}-2A_{1324}-A_{1414}=0,\quad
      A_{1212}+2A_{1313}-3A_{1324}-A_{1414}=0\,.
   \end{array}
$$
From this, we see that the matrix in Lemma \ref{lem-2.2} vanishes for $\varepsilon = -1$. This means that the
anti-self-dual Weyl
curvature operator $W_A^- $ vanishes so $\MM$ is self-dual. This is contrary to our assumption. Thus
for {\bf any} oriented orthonormal frame we have
\begin{equation}\label{eqn-2.a}
0=A_{1212}+2A_{1214}-2A_{1223}+2A_{1234}-A_{1324}+A_{1414}\,.
\end{equation}
Setting $\tilde e_1=-e_1$, $\tilde e_2=e_2$, $\tilde e_3=e_3$, and $\tilde e_4=-e_4$ yields
\begin{equation}\label{eqn-2.b}
0=A_{1212}-2A_{1214}+2A_{1223}+2A_{1234}-A_{1324}+A_{1414}\,.
\end{equation}
Subtracting Equation (\ref{eqn-2.b}) from Equation (\ref{eqn-2.a}) then yields the relation
$$0=-A_{1214}+A_{1223}\,.$$
We may now use Lemma \ref{lem-2.3} to complete the proof of Theorem \ref{thm-1.1}.

\section{The proof of Theorem \ref{thm-1.3}}\label{sect-3}

Here is a brief outline to this section. In Section \ref{sect-3.1}, we construct, up to isomorphism, all spacelike Jordan Osserman
models of signature $(2,2)$. In the remainder of Section \ref{sect-3}, we analyze each
possible Jordan normal form in some detail using the classification of Equation (\ref{eqn-1.e}). Sections
\ref{sect-3.2}-\ref{sect-3.5}
deal with Type Ia models. In Section
\ref{sect-3.2} we study the case when all the eigenvalues are equal; this gives rise to Theorem \ref{thm-1.3} (1).
In Section
\ref{sect-3.3}, we study the case of two equal
spacelike eigenvalues, and in Section \ref{sect-3.4}, we study equal timelike and spacelike eigenvalues; these involve Theorem
\ref{thm-1.3} (2) and (3), respectively. In Section
\ref{sect-3.5}, we study Type Ia models with distinct
eigenvalues; this leads to Theorem \ref{thm-1.3} (4).
We complete the proof of Theorem \ref{thm-1.3} by showing the remaining  Types do not give rise to null Jordan
Osserman models. Type Ib models are studied in Section
\ref{sect-3.6}, Type II models are studied in Section
\ref{sect-3.7}, and Type III models are studied in Section
\ref{sect-3.8}.

\subsection{Spacelike Jordan Osserman models}\label{sect-3.1} We use the ansatz from
\cite{G-I-2}. Let
$\{\Psi_1,\Psi_2,\Psi_3\}$ be the paraquaternionic structure given in Equation (\ref{eqn-1.d}). Let
$\xi_{ij}\in\mathbb{R}$ for
$1\le i\le j\le 3$ and let $\kappa_0\in\mathbb{R}$ be given. Let
\begin{eqnarray}\label{eqn-3.a}
&&
A_{\kappa_0,\xi}:=\kappa_0A^0+\textstyle\frac13\xi_{11}A^{\Psi_1}+\textstyle\frac13\xi_{22}A^{\Psi_2}
+\textstyle\frac13\xi_{33}A^{\Psi_3}\\
&&\qquad\quad+\textstyle\frac13\xi_{12}A^{\Psi_1+\Psi_2}+\textstyle\frac13\xi_{13}A^{\Psi_1+\Psi_3}
+\textstyle\frac13\xi_{23}A^{\Psi_2+\Psi_3},\nonumber\\
&&\mathcal{J}_{\kappa_0,\xi}:=\kappa_0\operatorname{id}+\left(\begin{array}{rrr}
\xi_{11}+\xi_{12}+\xi_{13}&-\xi_{12}&-\xi_{13}\\
\xi_{12}&-\xi_{22}-\xi_{12}-\xi_{23}&-\xi_{23}\\
\xi_{13}&-\xi_{23}&-\xi_{33}-\xi_{13}-\xi_{23}
\end{array}\right)\,.\nonumber\end{eqnarray}

\begin{lemma}\label{lem-3.1}
Adopt the notation established above. Let $\MM_{\kappa_0,\xi}:=(V,\langle\cdot,\cdot\rangle,A_{\kappa_0,\xi})$.
 \begin{enumerate}
\item If $x\in S^\pm(V,\langle\cdot,\cdot\rangle)$, then $\mathcal{J}_{A_{\kappa_0,\xi}}(x)$ is conjugate
to the matrix
$\pm\mathcal{J}_{\kappa_0,\xi}$.
\item The model $\MM_{\kappa_0,\xi}$ is spacelike and timelike Jordan Osserman.
\item Let $\MM_i=(V,\langle\cdot,\cdot\rangle,A_i)$ be spacelike Osserman models of signature $(2,2)$. If $\mathcal{J}_{A_1}(x)$ is
conjugate to $\mathcal{J}_{A_2}(x)$ for some $x\in S^\pm(V,\langle\cdot,\cdot\rangle)$, then there exists an isometry $\phi$ of
$(V,\langle\cdot,\cdot\rangle)$ so that $\phi^*A_2=A_1$.
\end{enumerate}\end{lemma}

\begin{remark}\label{rmk-3.2}
\rm Since any self-adjoint map of a signature $(2,1)$ vector space is conjugate to $\mathcal{J}_{\kappa_0,\xi}$ for some
$\{{\kappa_0,\xi}\}$, every spacelike Osserman model of signature $(2,2)$ is isomorphic to one given by Equation
(\ref{eqn-3.a}).\end{remark}

\begin{proof} We suppose $x$ is a spacelike unit vector as the timelike case is similar. Let $f_1:=\Psi_1x$, $f_2:=\Psi_2x$, and
$f_3:=\Psi_3x$. Then
$\{f_1,f_2,f_3\}$ is an orthonormal basis of  signature $(+,-,-)$ for $x^\perp$. Let
$\mathcal{J}:=\mathcal{J}_{A_{\kappa_0,\xi}}(x)$. We use Equation (\ref{eqn-1.c}) to see:
\begin{eqnarray*}
&&\mathcal{J}f_1=(\kappa_0+\xi_{11}+\xi_{12}+\xi_{13})f_1+\xi_{12}f_2+\xi_{13}f_3,\\
&&\mathcal{J}f_2=-\xi_{12}f_1+(\kappa_0-\xi_{22}-\xi_{12}-\xi_{23})f_2-\xi_{23}f_3,\\
&&\mathcal{J}f_3=-\xi_{13}f_1-\xi_{23}f_2+(\kappa_0-\xi_{33}-\xi_{13}-\xi_{23})f_3\,.
\end{eqnarray*}
Assertion (1) now follows; Assertion (2) follows from Assertion (1). Suppose that $\MM$ is a Type Ia spacelike
Osserman model so
$\mathcal{J}_A(x)=\operatorname{diag}[\alpha,\beta,\gamma]$ for any $x$ in $S^+(V,\langle\cdot,\cdot\rangle)$; choose the notation
so $ \operatorname{Ker}(\mathcal{J}_A(x)-\alpha\operatorname{id})$ is spacelike. It then follows from the
discussion in
\cite{BBR,GKV} that there exists an orthonormal basis
$\BB$ so that the non-zero components of the curvature tensor are given by:
$$
\begin{array}{ll}
A_{1221}=A_{4334}=\alpha,&A_{1331}=A_{2442}=-\beta,\\
A_{1441}=A_{3223}=-\gamma,&
A_{1234}=(-2\alpha+\beta+\gamma)/3,\\
A_{1423}=(\alpha+\beta-2\gamma)/3,&
A_{1342}=(\alpha-2\beta+\gamma)/3\,.
\end{array}$$
Similar forms exist for the other  Types of Equation (\ref{eqn-1.e}). Thus the Jordan normal form of
$\mathcal{J}_A(x)$ determines $A$ up to the action of $O(2,2)$. Assertion (3) follows.
\end{proof}

The following observation is immediate:
\begin{lemma}\label{lem-3.3}
A null Osserman model $\MM$ of signature $(2,2)$ is null Jordan Osserman if and only if the functions
$\operatorname{Rank}\{\mathcal{J}_A(\cdot)\}$ and $\operatorname{Rank}\{\mathcal{J}_A(\cdot)^2\}$ are constant
on $N(V,\langle\cdot,\cdot\rangle)$.
\end{lemma}

\subsection{Type Ia with all eigenvalues equal [$\alpha=\beta=\gamma$]}\label{sect-3.2}
We set $A=\kappa_0 A^0$. By Lemma \ref{lem-3.1}, the Jordan normal form is given by $\operatorname{diag}[
\kappa_0,\kappa_0,\kappa_0]$. If
$v\in N(V,\langle\cdot,\cdot\rangle)$, then
$\mathcal{J}_A(v)y=-\kappa_0\langle v,y\rangle v$ and hence $\MM$ is null Jordan Osserman.

\subsection{Type Ia with two equal spacelike eigenvalues [$\beta=\gamma$, $\alpha\ne\beta$]}\label{sect-3.3}
Let $J$ be an orthogonal almost complex structure on $V$ and let $A=\kappa_0A^0+\kappa_JA^J$. The Jordan normal form
is then given by $\operatorname{diag}[\kappa_0+3\kappa_J,\kappa_0,\kappa_0]$ which has the desired
form for suitably chosen $\kappa_0$ and $\kappa_J$ with
$\kappa_J\ne0$. Let
$v\in N(V,\langle\cdot,\cdot\rangle)$. We have
$$\mathcal{J}_{A}(v)y=-\kappa_0\langle v,y\rangle v+3\kappa_J\langle
y,Jv\rangle Jv\,.$$
Because
$J^2=-\operatorname{id}$, $v$ and
$Jv$ are linearly independent vectors. We note that $\langle v,v\rangle=\langle v,Jv\rangle=\langle Jv,Jv\rangle=0$. Consequently
$\mathcal{J}_{A}(v)v=\mathcal{J}_{A}(v)Jv=0$.  Since
$v^\perp$ and $Jv^\perp$ are distinct $3$-dimensional subspaces, we can choose $y$ so $\langle v,y\rangle=1$ and $\langle
Jv,y\rangle=0$. It now follows that $\mathcal{J}_{A}(v)y=-\kappa_0 v$ while $\mathcal{J}_{A}(v)Jy=
3\kappa_JJv$. Thus
$\mathcal{J}_{A}(v)$ has rank $2$ and $\mathcal{J}_{A}(v)^2=0$. This implies $A$ is null Jordan Osserman.

\subsection{Type Ia with equal timelike and spacelike  eigenvalues [$\alpha=\beta$, $\beta\ne\gamma$]}
\label{sect-3.4}
Let
$A=\kappa_0A^0+\kappa_PA^P$ where $\kappa_P\ne0$ and where $P$ is an adapted paracomplex structure; the Jordan normal form
is then given by
$\operatorname{diag}[\kappa_0,\kappa_0-3\kappa_P,\kappa_0]$ which has the desired form for suitably chosen
parameters. If $v\in N(V,\langle\cdot,\cdot\rangle)$, then
$$\mathcal{J}_{A}(v)y=-\kappa_0\langle v,y\rangle v+3\kappa_P\langle y,Pv\rangle Pv\,.$$
If $\kappa_0=0$, $\MM$ is null Jordan Osserman. Suppose $\kappa_0\ne0$. If $v=e_1+Pe_1$, then $Pv=v$ so
$\operatorname{Rank}\{\mathcal{J}_{A}(v)\}\le1$. On the other hand, if $v=e_1+e_4$, then $v$ and $Pv$ are linearly
independent so $\operatorname{Rank}\{\mathcal{J}_{A}(v)\}=2$ and $\MM$ is not null Jordan Osserman.

\subsection{Type Ia with three distinct eigenvalues}\label{sect-3.5}
Let $A:=\sum_i\kappa_iA^{\Psi_i}$ where $\{\Psi_1,\Psi_2,\Psi_3\}$ is the paraquaternionic structure of Equation (\ref{eqn-1.d});
the Jordan normal form is given by
$\operatorname{diag}[3\kappa_1,-3\kappa_2,-3\kappa_3]$ which has the desired form for suitably chosen
parameters with
$$\kappa_1+\kappa_2\ne0,\quad\kappa_1+\kappa_3\ne0,\quad\kappa_2-\kappa_3\ne0\,.$$
Let
$\tilde e\in S^+(V,\langle\cdot,\cdot\rangle)$, let
$V_+:=\operatorname{Span}\{\tilde e,\Psi^1\tilde e\}$, and let $
V_-=V_+^\perp=\operatorname{Span}\{\Psi_2\tilde e,\Psi_3\tilde e\}$.
We then have an orthogonal direct sum decomposition $V=V_-\oplus V_+$ where $V_+$ is spacelike and $V_-$ is timelike. Decompose
$v\in N(V,\langle\cdot,\cdot\rangle)$ in the form $v=\lambda(e_++e_-)$ where $e_\pm\in V_\pm$. Let $\MM$ be spacelike Osserman.
We have $\mathcal{J}_A(v)=\lambda^2\mathcal{J}_A(e_++e_-)$. Since $\mathcal{J}_A(v)$ is nilpotent, $\mathcal{J}_A(v)$ and
$\mathcal{J}_A(e_++e_-)$ have the same Jordan normal form. Thus we may safely take $\lambda=1$ so $v=e_++e_-$.
Set $e=e_+$ and expand $e_-=\cos\theta\Psi_2e+\sin\theta\Psi_3e$. This expresses
$$
v=e+\cos\theta\Psi_2e+\sin\theta\Psi_3e\quad\text{for}\quad e\in S^+(V,\langle\cdot,\cdot\rangle)\,.
$$
We use the relations
$\Psi_1\Psi_2=\Psi_3$,
$\Psi_1\Psi_3=-\Psi_2$, and
$\Psi_2\Psi_3=-\Psi_1$ to see
\begin{equation}\label{eqn-3.b}
\begin{array}{l}\begin{array}{rrrrrr}
\Psi_1v=&0&+\Psi_1e&-\sin\theta\Psi_2e&+\cos\theta\Psi_3e,\\
\Psi_2v=&\cos\theta e&-\sin\theta\Psi_1e&+\Psi_2e&+0,\\
\Psi_3v=&\sin\theta e&+\cos\theta\Psi_1e&+0&+\Psi_3e,\end{array}\\
\ \ 0=\Psi_1v+\sin\theta\Psi_2v-\cos\theta\Psi_3v\,.
\end{array}\end{equation}
This shows that the vectors $\{\Psi_1v,\Psi_2v,\Psi_3v\}$
span a $2$-dimensional subspace.
As $\langle\Psi_iv,\Psi_jv\rangle=0$, $\operatorname{Span}\{\Psi_iv\}\subset\operatorname{Ker}\{\mathcal{J}_A(v)\}$.
As $\operatorname{Range}\{\mathcal{J}_A(v)\}\subset\operatorname{Span}\{\Psi_iv\}$,
$$\operatorname{Rank}\{\mathcal{J}_A(v)\}\le2\quad\text{and}\quad\mathcal{J}_A(v)^2=0\,.$$
Note that $\{e,\Psi_1e,\Psi_2v,\Psi_3v\}$ is a basis for $V$. Let $\pi_+$ denote orthogonal projection on
$V_+=\operatorname{Span}\{e,\Psi_1e\}$. As $\pi_+$ is injective on
$\operatorname{Range}\{\mathcal{J}_A(v)\}\subset\operatorname{Span}\{\Psi_2v,\Psi_3v\}$,
$$r(v):=\dim\operatorname{Range}\{\mathcal{J}_A(v)\}=
\dim\{\operatorname{Span}\{\pi_+\mathcal{J}_A(v)e,\pi_+ \mathcal{J}_A(v)\Psi_1e\}\}\,.$$

By Equation (\ref{eqn-3.b}),
\begin{eqnarray*}
&&\mathcal{J}_A(v)e=3 \kappa_2\cos\theta\Psi_2v+3 \kappa_3\sin\theta\Psi_3v,\\
&&\mathcal{J}_A(v)\Psi_1e=3 \kappa_1\Psi_1v-3 \kappa_2\sin\theta\Psi_2 v+3\kappa_3\cos\theta\Psi_3v,\\
&&\pi_+\mathcal{J}_A(v)e=3 \{\kappa_2\cos\theta(\cos\theta)+\kappa_3\sin\theta(\sin\theta)\}e\\&&\qquad\qquad
    +3 \{\kappa_2\cos\theta(-\sin\theta)+\kappa_3\sin\theta(\cos\theta)\}\Psi_1e,\\
&&\pi_+\mathcal{J}_A(v)\Psi_1e=3 \{-\kappa_2\sin\theta(\cos\theta)+\kappa_3\cos\theta(\sin\theta)\}e\\&&\qquad\qquad
+3\{\kappa_1-\kappa_2\sin\theta(-\sin\theta)+\kappa_3\cos\theta(\cos\theta)\}\Psi_1e\,.
\end{eqnarray*}
This leads to a coefficient matrix for $\pi_+\mathcal{J}_A(v)$ on $V_+$ given by
$$\mathcal{C}_A(\theta)=3\left(\begin{array}{ll}
\kappa_2\cos^2\theta+\kappa_3\sin^2\theta&(-\kappa_2+\kappa_3)\sin\theta\cos\theta\\
(-\kappa_2+\kappa_3)\sin\theta\cos\theta&\kappa_1+\kappa_2\sin^2\theta+\kappa_3\cos^2\theta
\end{array}\right)\,.$$
We compute:
\begin{eqnarray*}
{\textstyle \frac{1}{9}}  \det(\mathcal{C}_A)(\theta)
&=&\kappa_1\kappa_2\cos^2\theta+\kappa_2^2\cos^2\theta\sin^2\theta+\kappa_2\kappa_3\cos^4\theta\\
&+&\kappa_1\kappa_3\sin^2\theta+\kappa_2\kappa_3\sin^4\theta+\kappa_3^2\sin^2\theta\cos^2\theta\\
&-&\kappa_2^2\sin^2\theta\cos^2\theta-\kappa_3^2\sin^2\theta\cos^2\theta+2\kappa_2\kappa_3\sin^2\theta\cos^2\theta\\
&=&\kappa_1\kappa_2\cos^2\theta+\kappa_1\kappa_3\sin^2\theta+\kappa_2\kappa_3\\
&=&(\kappa_1+\kappa_3)\kappa_2\cos^2\theta+(\kappa_1+\kappa_2)\kappa_3\sin^2\theta\,.
\end{eqnarray*}
Observe that $\kappa_2\kappa_3=0$ implies that
$\det(\mathcal{C}_A)(\theta)$ vanishes for some $\theta$ and thus
$\MM$ is not null Jordan Osserman. Hence, since
$(\kappa_1+\kappa_3)\kappa_2$ and $(\kappa_1+\kappa_2)\kappa_3$ are
non-zero, $\det(\mathcal{C}_A)(\theta)$ never vanishes, or
equivalently $\MM$ is null Jordan Osserman, if and only if these two
real numbers have the same sign, i.e.
$\kappa_2\kappa_3(\kappa_1+\kappa_3)(\kappa_1+\kappa_2)>0$.

\subsection{Type Ib models}\label{sect-3.6}
Let $b\ne0$. We take a curvature tensor of the form:
$$
A=\textstyle\frac13\{(a-b)A^{\Psi_1}+(-b-a)A^{\Psi_2}+bA^{\Psi_1+\Psi_2}+cA^{\Psi_3}\}\,.
$$
Proceeding as in the previous case, we have  for any $e\in S^+(V,\langle\cdot,\cdot\rangle)$ that:
\begin{eqnarray*}
&&\mathcal{J}_A(x)y=\langle(a\Psi_1+b\Psi_2)x,y\rangle\Psi_1x
   +\langle(b\Psi_1-a\Psi_2)x,y\rangle\Psi_2x+c\langle\Psi_3x,y\rangle\Psi_3x,\\
&&\mathcal{J}_A(e)\Psi_1e=a\Psi_1e+b\Psi_2e,\
  \mathcal{J}_A(e)\Psi_2e=-b\Psi_1e+a\Psi_2e,\
  \mathcal{J}_A(e)\Psi_3e=-c\Psi_3e\,.
\end{eqnarray*}
Thus
$\MM:=(V,\langle\cdot,\cdot\rangle,A)$ is Type Ib and any
Type Ib model is isomorphic to $\MM$ for suitably chosen
parameters. As in  Section \ref{sect-3.5}, put
$v=e+\cos\theta\Psi_2e+\sin\theta\Psi_3e$. We compute:
\medbreak\quad
$\mathcal{J}_A(v)e=b\cos\theta\Psi_1v-a\cos\theta\Psi_2v+c\sin\theta\Psi_3v$,
\smallbreak\quad $\mathcal{J}_A(v)
\Psi_1e=(a-b\sin\theta)\Psi_1v+(b+a\sin\theta)\Psi_2v+c\cos\theta\Psi_3v$,
\smallbreak\quad
$\pi_+\mathcal{J}_A(v)e=\{-a\cos\theta(\cos\theta)+c\sin\theta(\sin\theta)\}e$
\par\qquad
$+\{b\cos\theta-a\cos\theta(-\sin\theta)+c\sin\theta(\cos\theta)\}\Psi_1e$,
\smallbreak\quad
$\pi_+\mathcal{J}_A(v)\Psi_1e=\{(b+a\sin\theta)(\cos\theta)+c\cos\theta(\sin\theta)\}e$
\par\qquad
$+\{(a-b\sin\theta)+(b+a\sin\theta)(-\sin\theta)+c\cos\theta(\cos\theta)\}\Psi_1e$.
\medbreak\noindent The coefficient matrix for
$\pi_+\mathcal{J}_A(v)$ on $V_+$ is then given by
$$\mathcal{C}_A(\theta)=\left(\begin{array}{ll}
-a\cos^2\theta+c\sin^2\theta&b\cos\theta+(a+c)\sin\theta\cos\theta\\
b\cos\theta+(a+c)\sin\theta\cos\theta&-2b\sin\theta+(a+c)\cos^2\theta
\end{array}\right)\,.$$
We have $\det(\mathcal{C}_A)(\frac\pi2)=-2bc$, $\det(\mathcal{C}_A)(-\frac\pi2)=2bc$. If $c\ne0$, then these signs differ and hence
$\det(\mathcal{C}_A)(\theta)=0$ for some $-\frac\pi2<\theta<\frac\pi2$ and $\MM$ is not null Jordan Osserman. If $c=0$, then
$\det(\mathcal{C}_A)(\frac\pi2)=0$ and $\det(\mathcal{C}_A)(0)=-a^2-b^2\ne0$ and again $\MM$ is not null Jordan Osserman. This
completes the analysis in this setting.

\subsection{Type II models}\label{sect-3.7}
We take a direct approach to this case.
Let $\MM=(V,\langle\cdot,\cdot\rangle,A)$ be a model of signature $(2,2)$, where $A$ is a Type II
algebraic curvature tensor. Then the analysis of \cite{BBR,GKV} shows there exists  an  orthonormal
basis $\{ e_1,e_2,e_3,e_4\}$ for
$V$ such that
the non-vanishing components of $A$ are
$$
\begin{array}{l}
   A_{1221}=A_{4334}=\pm\left(\alpha-\frac{1}{2}\right),\quad
   A_{1331}=A_{4224}=\mp\left(\alpha+\frac{1}{2}\right), \\[0.025in]
   A_{1441}=A_{3223}=-\beta, \quad
   A_{2113}=A_{2443}=\mp\frac{1}{2},\quad
   A_{1224}=A_{1334}=\pm\frac{1}{2},\\[0.025in]
   A_{1234}=\left(\pm\left(-\alpha+\frac{3}{2}\right)+\beta\right)/3,
   \quad
   A_{1423}=2(\pm\alpha-\beta)/3,\\[0.025in]
   A_{1342}=\left(\pm\left(-\alpha-\frac{3}{2}\right)+\beta\right)/3\,.
\end{array}
$$
Let $u=e_2-e_3$ and let $v=e_2+e_3$. Then
$$
   \J_A(u) =
   \left(
   \begin{array}{cccc}
      0 & 0 & 0 & 0
      \\
      0 & \beta & \beta & 0
      \\
      0 & -\beta & -\beta & 0
      \\
      0 & 0 & 0 & 0
   \end{array}\right)\quad\text{and}\quad
   \J_A(v) =
   \left(
   \begin{array}{cccc}
      \pm 2 & 0 & 0 & \mp 2
      \\
      0 & \beta & -\beta & 0
      \\
      0 & \beta & -\beta & 0
      \\
      \pm 2 & 0 & 0 & \mp 2
   \end{array}
   \right)\,.$$
If $\beta=0$, then $r(u)=0$ and $r(v)=1$; if $\beta\ne0$, then $r(u)=1$ and $r(v)=2$. Thus $\MM$ is not null Jordan Osserman.

\subsection{Type III models}\label{sect-3.8}
If $\MM$ is Type III, then
there exists an orthonormal basis $\{ e_1,e_2,e_3,e_4\}$ for
$V$  such
that the non-vanishing components of $A$ are (see \cite{BBR,GKV})
$$
\begin{array}{l}
   A_{1221}=A_{4334}=\alpha,\quad
   A_{1331}=A_{4224}=-\alpha,\quad
   A_{1441}=A_{3223}=-\alpha, \\
   \noalign{\medskip}
   A_{2114}=A_{2334}=-\sqrt{2}/2,\quad
   A_{3114}=-A_{3224}=\sqrt{2}/2, \\
   \noalign{\medskip}
   A_{1223}=A_{1443}=A_{1332}=-A_{1442}=\sqrt{2}/2.
   \end{array}
$$
Let $u=e_2-e_3$ and $v=e_2+e_3$. Then:
$$
\J_A(u) =
   \left(
   \begin{array}{cccc}
      0 & -\sqrt{2} & -\sqrt{2} & 0
      \\
      -\sqrt{2} & \alpha & \alpha & \sqrt{2}
      \\
      \sqrt{2} & -\alpha & -\alpha & -\sqrt{2}
      \\
      0 & -\sqrt{2} & -\sqrt{2} & 0
   \end{array}
   \right)\text{ and }
\J_A(v) =
   \left(
   \begin{array}{cccc}
      0 & 0 & 0 & 0
      \\
      0 & \alpha & -\alpha & 0
      \\
      0 & \alpha & -\alpha & 0
      \\
      0 & 0 & 0 & 0
   \end{array}
   \right)\,.$$
It now follows that $r(u)=2$ while $r(v)\le1$ and hence $\MM$ is not null Jordan Osserman. This completes the proof of Theorem
\ref{thm-1.3}.

\section{The proof of Theorem \ref{thm-1.4}}\label{sect-4}

Let $\mathcal{M}$ be a null Jordan Osserman manifold of signature $(2,2)$. First note that, by Theorem
\ref{thm-1.3}, $\mathcal{M}$ has Type Ia. Results of
\cite{BBR} then show that $\mathcal{M}$ either (a) has constant sectional curvature, (b) is locally isometric
to a complex space form, or (c) is locally isometric to a paracomplex space form. Since
the curvature tensor of a paracomplex space
form of constant paraholomorphic sectional curvature $\kappa$ satisfies
$$
   R(x,y)z
   =\mbox{$\frac{\kappa}{4}$} \left\{ R^0(x,y)z - R^J(x,y)z \right\};
$$
this is ruled out by Theorem \ref{thm-1.3}. This completes the proof of Theorem \ref{thm-1.4}

\section*{Acknowledgments}Research of E. Garc\'{\i}a-R\'{\i}o, M. E. V\'{a}zquez-Abal, and  R.
V\'{a}zquez-Lorenzo supported by projects  MTM2006-01432 and  PGIDIT06PXIB207054PR (Spain). Research of
P. Gilkey partially supported by the Max Planck Institute in the Mathematical Sciences (Leipzig, Germany)
and by Project MTM2006-01432 (Spain).

\end{document}